\documentclass[12pt,a4paper]{article}
\usepackage[utf8]{inputenc}
\usepackage[T2A]{fontenc}
\usepackage{amscd,amssymb}
\usepackage{amsfonts,amsmath,array}
\usepackage[dvips]{graphicx}
\usepackage{longtable,wrapfig}
\usepackage{amsthm}
\usepackage{xcolor}
\newtheorem{mydef}{Definition}

\newtheorem{theorem}{Theorem}

\newtheorem{corollary}{Corollary}
\newtheorem{conj}{Conjecture}

\newtheorem{lemma}{Lemma}
\newcommand{\m}{\mathcal}

\title{Families without $s$-matchings: the other end}
\author{Andrey Kupavskii\footnote{Moscow Institute of Physics and Technology, St. Petersburg State University, Innopolis University, Russia; Email: {\tt kupavskii@ya.ru}}, Georgy Sokolov\footnote{Moscow Institute of Physics and Technology, Innopolis University, Russia; Email: {\tt sokolov.gm@phystech.edu}}}
\begin{document}
	\maketitle
\begin{abstract}
    In this paper, we determine the largest family $\mathcal F \subset 2^{[n]}$ without $s$ pairwise disjoint sets, provided  $n=ms+c$ for positive integers $m,c$, and $s \geq s_0(m, c)$. This result can be seen as a non-uniform analogue of the results on the Erd\H os Matching Conjecture in the regime when the clique is extremal. 
\end{abstract}
\section{Introduction}

We begin with some notation. For integers $a\le b$, let
$[a,b]:=\{a,a+1,\ldots,b\}$ and $[b]:=[1,b].$
For a set $X$ and an integer $k$, we write $2^X$ for the power set of $X$, ${X\choose k}$ for the family of all $k$-element subsets of $X$, and ${X\choose \le k}$ for the family of all subsets of $X$ of size at most $k$.

A {\it matching} is a collection of pairwise disjoint sets. For a family $\mathcal F$, let $\nu(\mathcal F)$ denote the maximum size of a matching contained in $\mathcal F$. The main object of this paper is the quantity
\[
e(n,s):=\max\bigl\{|\mathcal F|:\mathcal F\subset 2^{[n]},\ \nu(\mathcal F)<s\bigr\}.
\]
This parameter was introduced by Erd\H{o}s and Kleitman in the 1960s (see \cite{Kl}). It was studied further in a series of papers by Frankl and the first author~\cite{FK9,FK6,FK8}, as well as in our earlier paper \cite{KS}. For a detailed and up-to-date account of the history of the problem and its connection to the Erd\H{o}s Matching Conjecture \cite{E}, we refer the reader to \cite{KS}.

A basic feature of the problem is that the behaviour of $e(n,s)$ depends strongly on the residue class of $n$ modulo $s$. Accordingly, we write
\[
n=ms+c, \qquad 1\le c\le s,
\]
and throughout the paper set
\[
\ell:=s-c.
\]

We now introduce four candidate extremal families avoiding $s$-matchings:
\begin{align*}
\mathcal P(m,s,\ell) &:= \{P\in 2^{[n]}: |P|+|P\cap [\ell-1]|\ge m+1\},\\
\mathcal P'(m,s,\ell) &:= \bigcup_{i=m+1}^n {[n]\choose i}\ \cup\ {[m\ell-1]\choose m},\\
\mathcal Q(m,s,\ell) &:= \{P\in 2^{[n]}: |P|+|P\cap [ms-c-1]|\ge 2m\},\\
\mathcal W(m,s,\ell) &:= \{P\in 2^{[n]}: |P\cap [ms-1]|\ge m\}.
\end{align*}

The families $\mathcal P(m,s,\ell)$ and $\mathcal P'(m,s,\ell)$ contain every set of size at least $m+1$, and their $m$-uniform layers are precisely the two standard candidate extremal families in the Erd\H{o}s Matching Conjecture for $m$-uniform families avoiding an $\ell$-matching.

The third family $\m Q(m,s,\ell)$ can be thought of as another  `clique' counterpart of the family $\m P(m,s,\ell)$, the first being $\m P'(m,s,\ell)$. It has a larger clique on the $m$'th layer in exchange for fewer sets on the upper layers.

The family $\mathcal W(m,s,\ell)$ is obtained by applying the {\it doubling} operation $c+1$ times to the extremal family for $n=ms-1$. Doubling turns a family $\mathcal F\subset 2^{[n]}$ into the family $\mathcal G\subset 2^{[n+1]}$ defined by
$
\mathcal G:=\{A\subset [n+1]: A\cap [n]\in \mathcal F\}.
$
This operation preserves the matching number. Kleitman proved in \cite{Kl} that $e(ms-1,s), e(ms,s)$ is attained on $\m W(m,s,\ell)$. (In the former case, it is simply the family of all sets of size at least $m$.)

In the next section we discuss the sizes of these four families and show that each of them indeed avoids an $s$-matching.

For every pair $(n,s)$ for which the value of $e(n,s)$ is currently known, the extremal family is one of the four constructions above. In particular, in \cite{KS} we  showed that for $n\le 3s$ each of these four families occur as extremal examples. Roughly speaking, when $c$ is close to $s$, the extremal construction is of type $\mathcal P(m,s,\ell)$, whereas for small $c$ the known extremal examples are of the type $\mathcal Q(m,s,\ell)$ or $\mathcal W(m,s,\ell)$. The family $\m P'(m,s,\ell)$ appears as the extremum somewhere in the middle.

This leads to the following general conjecture.

\begin{conj}
For all $n,s$,
\[
e(n,s)=\max\bigl\{|\mathcal P(m,s,\ell)|,\ |\mathcal P'(m,s,\ell)|,\ |\mathcal Q(m,s,\ell)|,\ |\mathcal W(m,s,\ell)|\bigr\}.
\]
\end{conj}

In \cite{FK9}, the authors confirmed it for $s\ge m\ell+3\ell+3$, showing that $\m P(m,s,\ell)$ is extremal. The main result of this paper confirms the conjecture `on the other end': for fixed $m,c$ and all sufficiently large $s$.

\begin{theorem} \label{t.main}
For every $m\ge 1$ and every $c\ge 1$ there exists $s_0$ such that for all $s\ge s_0$,
\[
e(n,s)=|\mathcal Q(m,s,\ell)|.
\]
Moreover,
\begin{itemize}
    \item if $c\ge 2$, then $\mathcal Q(m,s,\ell)$ is the unique shifted family attaining the maximum;
    \item if $c=1$, then $\mathcal Q(m,s,\ell)$ and $\mathcal W(m,s,\ell)$ are the only shifted families attaining the maximum.
\end{itemize}
\end{theorem}
This result can be seen as a non-uniform analogue of the results on the Erd\H os Matching Conjecture in the regime when the clique is extremal, see \cite{Frankl_perfect_EMC, KoKu}. 

We recall the notion of shiftedness at the beginning of Section~\ref{sec3}. Throughout the paper, asymptotic notation (such as $O(), \Theta()$) is used with $m$ and $c$ fixed and $s,n\to\infty$. Since the value of $e(ms+c,s)$ is already known for $m=1$ by \cite{FK9}, from this point on we assume $m\ge 2$.

Finally, we note that when $c=1$, the families $\mathcal Q(m,s,\ell)$ and $\mathcal W(m,s,\ell)$ coincide. This will be shown in the next section, where we also discuss further properties of the extremal families.

\section{Extremal families}
The families that we consider contain almost all subsets of $[n]$. Thus, rather than computing the sizes of the families themselves, it is more convenient to compute the sizes of their complements. For a family $\mathcal{F} \subset 2^{[n]}$ we denote $\overline{\mathcal{F}}:=2^{[n]} \setminus \mathcal{F}.$ 
Simple computations show that
\begin{align*}
|\overline{\mathcal{P}(m, s, \ell)}| &= \Theta(s^m),\\
|\overline{\mathcal{P}'(m, s, \ell)}| &= ((m+1)c+2){n \choose m - 1} + O(s^{m-2}),\\
|\overline{\mathcal{Q}(m, s, \ell)}| &= (2c+2){n \choose m - 1} + O(s^{m-2}),\\
|\overline{\mathcal{W}(m, s, \ell)}| &= 2^{c+1}{n \choose m - 1} + O(s^{m-2}).
\end{align*}

The following simple observation is useful for proving upper bounds on the matching number.

\begin{lemma} \label{l.weak_duality}
    Let $x_1, \ldots, x_n\in \mathbb R$ satisfy $x_i \geq 0$ and $\sum_{i=1}^{n}x_i < s$. Let $\mathcal{F}$ be a family such that for every $F \in \mathcal{F}$ we have $\sum_{i\in F}x_i \geq 1$. Then $\nu(\mathcal{F}) < s$.
\end{lemma}

\begin{proof}
    We argue indirectly. Let $F_1, \ldots, F_s$ be an $s$-matching in $\mathcal{F}$. Then
    \[
    \sum_{i=1}^{n}x_i \geq \sum_{i\in \bigcup_{j = 1}^s F_j}x_i = \sum_{j=1}^{s}\sum_{i\in F_j}x_i \geq s.
    \]
    This contradicts the assumption $\sum_{i=1}^{n}x_i < s$.
\end{proof}

\begin{corollary}
    The families $\mathcal{P}(m, s, \ell)$, $\mathcal{P'}(m, s, \ell)$, $\mathcal{Q}(m, s, \ell)$, and $\mathcal{W}(m, s, \ell)$ do not contain an $s$-matching.
\end{corollary}

\begin{proof}
    We apply Lemma~\ref{l.weak_duality}. Specifically, we choose the weights $x_i$ as follows:
    \begin{itemize}
        \item[] $\mathcal{P}(m, s, \ell)$: put $x_i = \frac{2}{m+1}$ for $i \leq \ell - 1$ and $x_i = \frac{1}{m+1}$ for $i \geq \ell$;
        \item[] $\mathcal{P}'(m, s, \ell)$: put $x_i = \frac{1}{m}$ for $i \leq m\ell-1$ and $x_i = \frac{1}{m+1}$ for $i \geq m\ell$;
        \item[] $\mathcal{Q}(m, s, \ell)$: put $x_i = \frac{1}{m}$ for $i \leq ms - c - 1$ and $x_i = \frac{1}{2m}$ for $i \geq ms-c$;
        \item[] $\mathcal{W}(m, s, \ell)$: put $x_i = \frac{1}{m}$ for $i \leq ms - 1$ and $x_i = 0$ for $i \geq ms$.
    \end{itemize}
\end{proof}
    \begin{lemma} \label{l.c=1 equality}
     We have   $|\mathcal{Q}(m, s, s-1)| = |\mathcal{W}(m, s, s-1)|$.
    \end{lemma}
\begin{proof}
We will calculate $|\overline{\mathcal{Q}(m, s, s-1)}|$ and $|\overline{\mathcal{W}(m, s, s-1)}|$.
Since $\mathcal{W}(m, s, s-1)$ is obtained from ${[ms - 1] \choose \geq m}$ by iterating the doubling operation twice, we have
\begin{equation} \label{eq: wsz_c=1}
    |\overline{\mathcal{W}(m, s, s-1)}| = 4\sum_{i=0}^{m-1}{ms-1 \choose i}.
\end{equation}

To calculate $|\overline{\mathcal{Q}(m, s, s-1)}|$, note that
\[
\overline{\mathcal{Q}(m, s, s-1)} = \{F \in 2^{[sm+1]}: 2|F| - |F \cap \{ms-1, ms, ms+1\}| < 2m \}.
\]
Thus, $\overline{\mathcal{Q}(m, s, s-1)}$ consists of
\begin{itemize}
    \item all sets of size less than $m$. There are $\sum_{i=0}^{m - 1}{sm+1 \choose i}$ such sets;
    \item sets of size $m$ that intersect $\{sm - 1, sm, sm+1\}$. There are ${sm + 1 \choose m} - {sm - 2 \choose m}$ such sets;
    \item sets of size $m + 1$ that contain $\{sm - 1, sm, sm + 1\}$. There are ${sm - 2 \choose m - 2}$ such sets.
\end{itemize}

Therefore,
\begin{equation} \label{eq: qsz_c=1}
    |\overline{\mathcal{Q}(m, s, s-1)}|
    = \sum_{i=0}^{m}{sm+1 \choose i} - {sm - 2 \choose m} + {sm - 2 \choose m - 2}.
\end{equation}

To prove that \eqref{eq: qsz_c=1} is equal to \eqref{eq: wsz_c=1}, we express ${sm+1 \choose i}$ in terms of ${sm-1 \choose i}$:
\[
{sm+1 \choose i} = {sm-1 \choose i} + 2 {sm-1 \choose i - 1} + {sm-1 \choose i - 2},
\]
where, as usual, we set ${n \choose k} = 0$ for $k < 0$. Summing over $i \in \{0, \ldots, m\}$, we get
\begin{align*}
    \sum_{i=0}^{m}{sm+1 \choose i}
    &= \sum_{i=0}^{m}{sm-1 \choose i} + 2\sum_{i=1}^{m}{sm-1 \choose i-1} + \sum_{i=2}^{m}{sm-1 \choose i-2} \\
    &= 4\sum_{i=0}^{m-2}{sm-1 \choose i} + 3{sm-1 \choose m-1} + {sm-1 \choose m} \\
    &= 4\sum_{i=0}^{m-1}{sm-1 \choose i} - {sm-1 \choose m-1} + {sm-1 \choose m}.
\end{align*}

Finally, substituting this into \eqref{eq: qsz_c=1}, we obtain
{\small
\begin{align*}
    |\overline{\mathcal{Q}(m, s, s-1)}|
    &= 4\sum_{i=0}^{m-1}{sm-1 \choose i} - {sm-1 \choose m-1} + {sm-1 \choose m} - {sm - 2 \choose m} + {sm - 2 \choose m - 2} \\
    &= 4\sum_{i=0}^{m-1}{sm-1 \choose i} + \left({sm-1 \choose m} - {sm - 2 \choose m}\right) - \left({sm-1 \choose m-1} - {sm - 2 \choose m - 2}\right) \\
    &= 4\sum_{i=0}^{m-1}{sm-1 \choose i} + {sm-2 \choose m-1} - {sm-2 \choose m-1} \\
    &= 4\sum_{i=0}^{m-1}{sm-1 \choose i}
     = |\overline{\mathcal{W}(m, s, s-1)}|.
\end{align*}
}
\end{proof}
	\section{Overview of the proof}\label{sec3}
Let us recall the notion of shifting. We will denote a set $\{a_1, \ldots, a_k\}$ by $(a_1, \ldots, a_k)$ if $a_1 < \ldots < a_k$ and we want to emphasize that its elements are listed in increasing order. We say that a set $(a_1, \ldots, a_k)$ can be shifted to a set $(b_1, \ldots, b_k)$ if $a_i \geq b_i$ for all $i \in [k]$. A family $\mathcal{F}$ is called shifted if, whenever $A \in \mathcal{F}$, we also have $B \in \mathcal{F}$ for every $B$ such that $A$ can be shifted to $B$. It is well known \cite{Frankl_Shifting} that the maximum size of a family of subsets of an $n$-element set avoiding an $s$-matching is attained by a shifted family. Therefore, it is enough to consider shifted families.

    Let $\mathcal{F} \subset 2^{[n]}$ be a family of subsets of $[n]$.
    We denote by $\mathcal{F}^{(i)}$ the $i$-th layer of $\mathcal{F}$, that is, the family $\mathcal{F} \cap {[n] \choose i}$. Let $y_{\mathcal{F}}(i)$ be the number of $i$-element sets, which are not in $\mathcal{F}$, that is, $y_{\mathcal{F}}(i) = {n \choose i} - |\mathcal{F}^{(i)}|$.

The main part of the proof is the analysis of $y(m)$ and $y(m+1)$. In Section~\ref{s.ym_ym+1} we use a counting argument to prove that $y(m) + y(m+1)$ is minimized by the family $\mathcal{Q}(m, s, \ell)$. Moreover, we establish a stability statement showing that if $\mathcal{F}$ is a shifted family avoiding an $s$-matching and $\mathcal{F} \nsubseteq \mathcal{Q}(m, s, \ell)$, then
\[
y_{\mathcal{F}}(m) + y_{\mathcal{F}}(m+1)
\ge
y_{\mathcal{Q}(m,s,\ell)}(m) + y_{\mathcal{Q}(m,s,\ell)}(m+1) + \Theta(s^{m-1}).
\]

Since the family $\mathcal{Q}(m, s, \ell)$ contains all sets of size at least $m+2$, except for $O(s^{m-2})$ sets, this stability result immediately implies that $\mathcal{Q}(m, s, \ell)$ is the largest family avoiding an $s$-matching among all families that do not contain sets of size exactly $m-1$. In Section~\ref{s.no_m-1} we prove that any family containing an $(m-1)$-element set is smaller than $\mathcal{Q}(m, s, \ell)$. To this end, we first use the stability result for $y(m) + y(m+1)$ to show that $e(n, s)$ cannot be much larger than $|\mathcal{Q}(m, s, \ell)|$. Then, assuming that a family $\mathcal{F}$ with $\nu(\mathcal{F}) < s$ contains some $(m-1)$-element set $X$, we consider the family $\mathcal{F}' = \mathcal{F} \cap 2^{[n]\setminus X}$ and using the stability statement conclude that $\overline{\mathcal F'}$ is large, which in turn implies that $\mathcal F$ is smaller than $\mathcal{Q}(m, s, \ell)$.
%Then, assuming that a family $\mathcal{F}$ with $\nu(\mathcal{F}) < s$ contains some $(m-1)$-element set $X$, we consider the family $\mathcal{F}' = \mathcal{F} \cap 2^{[n]\setminus X}$ and using the previous statement conclude that $\overline{\mathcal F'}$ is large, which in turn implies that $\mathcal F$ is smaller than $\mathcal{Q}(m, s, \ell)$.

\section{Missing sets on the layers $m$ and $m+1$} \label{s.ym_ym+1}

We have the following for the prospective extremal family:
\[
y_{\mathcal{Q}(m, s, \ell)}(m) + y_{\mathcal{Q}(m, s, \ell)}(m + 1) = (2c+1){n \choose m - 1} + O(s^{m - 2}).
\]
In this section, we prove a stability theorem for $y(m) + y(m + 1)$.

% \begin{theorem}
%     Let $n = ms + c, c \geq 2, m \geq 2$. Let $\mathcal{F} \subset 2^{[n]}$ be a shifted family and $\nu(\mathcal{F}) < s$. If $\mathcal{F} \nsubseteq \mathcal{Q}(m, s, \ell)$, then $y_{\mathcal{F}}(m) + y_{\mathcal{F}}(m + 1) \geq (2c+2){n \choose m - 1} + O(s^{m - 2})$.
% \end{theorem}

\begin{theorem} \label{t.layer_stability}
    Let $n = ms + c$, $m \geq 2$. Let $\mathcal{F} \subset 2^{[n]}$ be a shifted family and $\nu(\mathcal{F}) < s$. Then
    \[
    y_{\mathcal{F}}(m) + y_{\mathcal{F}}(m + 1) \geq (2c+2){n \choose m - 1} + O(s^{m - 2})
    \]
    if one of the following holds:
    \begin{enumerate}
        \item $c \geq 2$ and $\mathcal{F} \nsubseteq \mathcal{Q}(m, s, \ell)$;
        \item $c = 1$, $\mathcal{F} \nsubseteq \mathcal{Q}(m, s, \ell)$, and $\mathcal{F} \nsubseteq \mathcal{W}(m, s, \ell)$.
    \end{enumerate}
\end{theorem}

Note that the assumption $\mathcal{F} \nsubseteq \mathcal{W}(m, s, \ell)$ in the second case of the theorem is crucial, since $\mathcal{W}(m, s, \ell) \nsubseteq \mathcal{Q}(m, s, \ell)$ and
\[
y_{\mathcal{W}(m, s, s-1)}(m) + y_{\mathcal{W}(m, s, s-1)}(m+1) = 3{n \choose m - 1} + O(s^{m - 2}).
\]

Theorem~\ref{t.layer_stability} immediately implies that for $c \geq 2$ and $s \geq s_0(m, c)$ the family $\mathcal{Q}(m, s, \ell)$ has the smallest possible value of $y(m) + y(m + 1)$ among families avoiding an $s$-matching. However, Theorem~\ref{t.layer_stability} does not immediately imply Theorem~\ref{t.main}, even for $c \geq 2$, because $\overline{\mathcal{Q}(m, s, \ell)}$ contains ${n \choose m - 1}$ sets of size $m - 1$.

To prove Theorem~\ref{t.layer_stability}, we bound $y(m)$ and $y(m+1)$ separately in terms of the value $d(\mathcal{F})$, defined below.

\begin{mydef}
    The \textit{deletion number} $d(\mathcal{F})$ is the minimum number of elements of the ground set $[n]$ that must be deleted in order to avoid an $\ell$-matching in $\mathcal{F}^{(m)}$. Formally,
    \[
    d(\mathcal{F}) = \min\Big\{|X| : \nu\Big(\mathcal{F} \cap {[n] \setminus X \choose m}\Big) < \ell\Big\}.
    \]
\end{mydef}

Note that if $\nu(\mathcal{F}^{(m)}) < \ell$, then $d(\mathcal{F}) = 0$.

We will use the following simple lemma.

\begin{lemma} \label{l.deleteon_with_shifting}
    If $\mathcal{F}$ is shifted, $\ell \geq 1$, and $d(\mathcal{F}) > d$, then for any integer $k$, $0 \leq k \leq \frac{d}{m}$, and any $X \subset [d + m\ell]$ with $|X| = d - km$, the family $\mathcal{F} \cap {[d+m\ell] \setminus X \choose m}$ contains an $(\ell+k)$-matching.
\end{lemma}

\begin{proof}
Assume that for some $X$ of size $d-km$ the family $\mathcal{F}':=\mathcal{F}\cap {[d+m\ell]\setminus X\choose m}$ does not contain an $(\ell+k)$-matching. Take a largest matching $\mathcal{M}:=\{X_1,\ldots, X_{\ell'}\}$ in $\mathcal{F}'$ and consider the set
\[
Y = X\cup X_1\cup\ldots\cup X_{k'},
\]
where $k' = \min \{k,\ell'\}$. Then, since $\mathcal{M}$ is a largest matching, the family $\mathcal{F}\cap {[d+m\ell]\setminus Y\choose m}$ does not contain an $\ell$-matching. At the same time, $|Y| \le d$.

Now, if $\mathcal{F}\cap {[n]\setminus Y\choose m}$ contains an $\ell$-matching, then, using the fact that $\mathcal{F}$ is shifted, we can obtain an $\ell$-matching in $[d+m\ell]\setminus Y$. We conclude that $\mathcal{F}\cap {[n]\setminus Y\choose m}$ does not contain an $\ell$-matching. This contradicts the assumption that $d(\mathcal{F})>d$.
\end{proof}

In the subsequent calculations we will use the following simple and standard estimate, which we state without proof.

\begin{lemma} \label{l.binom_asymp}
    Let $n = ms + c$, where $m$ and $c$ are constant and $s \rightarrow \infty$. Then for any nonnegative integer functions $f(m, c)$ and $g(m, c)$ we have
    \[
    {n - f(m, c) \choose m} - {n - f(m, c) - g(m, c) \choose m}
    = g(m,c){n \choose m - 1} + O(n^{m-2}).
    \]
\end{lemma}

We will also use a result on the Erd\H{o}s Matching Conjecture in the case $n \leq (k + \varepsilon(k))s$, namely, the following theorem due to Frankl from \cite{Frankl_perfect_EMC}.

\begin{theorem} \label{t.frankl_EMC}
    For every $k \geq 2$ there exists a positive $\varepsilon(k)$ such that for $ks \leq n \leq (k + \varepsilon(k))s$ the family ${[ks-1] \choose k}$ is the largest family in ${[n] \choose k}$ without an $s$-matching.
\end{theorem}

The proof in \cite{Frankl_perfect_EMC} gives the bound $\varepsilon(k) = \frac{1}{2k^{2k+1}}$. Note that in \cite{KoKu},  Kolupaev and the first author obtained a much better dependence, namely $\varepsilon(k) = \frac{1}{100k}$ for $k \geq 5$.

The next lemma bounds $y_{\mathcal{F}}(m)$ in terms of $d(\mathcal{F})$.

\begin{lemma} \label{l.d_to_ym}
    Let $m \geq 2$ and $n = ms+c$. If $\mathcal{F} \subset 2^{[n]}$ is a family with $d(\mathcal{F}) \leq d$, then
    \[
    y_{\mathcal{F}}(m) \geq ((m+1)c-d+1){n \choose m - 1} + O(s^{m - 2}).
    \]
\end{lemma}

\begin{proof}
    Let $X$ be a set of size $d$ such that $\nu(\mathcal{F} \cap {[n] \setminus X \choose m}) < \ell$, and let
    \[
    \mathcal{F'} = \mathcal{F} \cap {[n] \setminus X \choose m}.
    \]
    Since $m$ and $c$ are constants and $s$ tends to infinity, we may assume that
    $
    s > \frac{c(m+1)}{\varepsilon(m)} + c,
    $
    which is equivalent to $n < \ell(m + \varepsilon(m))$. Therefore, by Theorem~\ref{t.frankl_EMC} we get
    \[
    |\mathcal{F'}| \leq {m\ell - 1 \choose m}.
    \]

    Since
    \[
    \mathcal{F}^{(m)} \setminus \mathcal{F'} \subset {[n] \choose m} \setminus {[n] \setminus X \choose m},
    \]
    we obtain
    \[
    |\mathcal{F}^{(m)}| \leq {n \choose m} - {n - d \choose m} + {m\ell - 1 \choose m}.
    \]
    In particular,
    \[
    |\mathcal{F}^{(m)}| \leq d{n \choose m - 1} + {m\ell - 1 \choose m},
    \]
    and by Lemma~\ref{l.binom_asymp},
    \begin{align*}
        y_{\mathcal{F}}(m)
        &\geq {n \choose m} - d{n \choose m - 1} - {m\ell - 1 \choose m} \\
        &= \left({n \choose m} - {n - (m+1)c - 1 \choose m}\right) - d{n \choose m - 1} \\
        &= ((m+1)c-d+1){n \choose m - 1} + O(s^{m - 2}).
    \end{align*}
\end{proof}

Lemma~\ref{l.d_to_ym} implies that if $d(\mathcal{F}) < (m-1)c$, then
\[
y_{\mathcal{F}}(m) \geq (2c+2){n \choose m - 1} + O(s^{m - 2}).
\]
This bound alone is sufficient to prove Theorem~\ref{t.layer_stability} in this case. The next two lemmas allow us to lower bound the contribution of $y(m+1)$ in the case when  $d(\mathcal{F}) > (m-1)c$ and $d(\mathcal{F}) = (m-1)c$, respectively.

\begin{lemma} \label{l.d_geq_mc-c+1}
    If $\mathcal{F} \subset 2^{[n]}$ is a shifted family with $d(\mathcal{F}) > (m-1)c$, then
    \[
    y_{\mathcal{F}}(m+1) \geq (2c-1){n \choose m - 1} + O(s^{m - 2}).
    \]
\end{lemma}

\begin{proof}
    By Lemma \ref{l.deleteon_with_shifting} for $k=0$, for any
    \[
    X \in {[m\ell+(m-1)c] \choose (m-1)c} = {[ms-c] \choose (m-1)c},
    \]
    the family $\mathcal{F} \cap {[ms-c] \setminus X \choose m}$ contains an $\ell$-matching. For each $X$, fix one such matching $\pi_{X}$.

    Put 
    \[ \m G:={[ms-c] \choose m-1} \times {[ms-c+1, n] \choose 2}.\]
    Construct a random $s$-matching $\pi$ in
    $
    \big(\mathcal{F} \cap {[ms-c] \choose m}\big) \cup \m G
    $
    by the following procedure.
    \begin{itemize}
        \item Choose $X \in {[ms-c] \choose (m-1)c}$ uniformly at random.
        \item Add $\pi_{X}$ to $\pi$. Note that all sets added to $\pi$ in this step belong to $\mathcal{F}$.
        \item Choose a $c$-matching in ${X \choose m-1} \times {[ms-c+1, n] \choose 2}$ uniformly at random and add it to $\pi$. Note that such matchings exist, since $|X|=(m-1)c$ and $|[ms-c+1, n]| = 2c$.
    \end{itemize}

    Let $\xi$ be the number of sets from $\overline{\mathcal{F}}$ in $\pi$. On the one hand, $\xi \geq 1$, since $\pi$ is an $s$-matching and $\mathcal{F}$ does not contain an $s$-matching. Therefore, $\mathbb{E}\xi \geq 1$. On the other hand, for any set $F \in \m G$ we have
    \[
    \mathbb{P}[F \in \pi] = c{ms-c \choose m-1}^{-1}{2c \choose 2}^{-1}.
    \]
    Therefore,
    \[
    1\le \mathbb{E}\xi = |\overline{\mathcal{F}} \cap \m G| \cdot c{ms-c \choose m-1}^{-1}{2c \choose 2}^{-1}.
    \]
    Combining these inequalities, we get
    \[
    \begin{split}
        y_{\mathcal{F}}(m+1)
        &\geq |\overline{\mathcal{F}} \cap \m G| 
        \geq c^{-1}{ms-c \choose m-1}{2c \choose 2} \\
        &= (2c-1){n-2c \choose m-1} 
        = (2c-1){n \choose m - 1} + O(s^{m - 2}).
    \end{split}
    \]
\end{proof}

\begin{lemma} \label{l.stability_d_eq_mc-c}
    If $\mathcal{F} \subset 2^{[n]}$ is a shifted family with $d(\mathcal{F}) = (m-1)c$ and $\mathcal{F} \nsubseteq \mathcal{Q}(m,s,\ell)$, then
    \[
    y_{\mathcal{F}}(m+1) \geq {n \choose m - 1} + O(s^{m-2}).
    \]
\end{lemma}

\begin{proof}
    The proof is very similar to that of Lemma~\ref{l.d_geq_mc-c+1}, although the construction of a random $s$-matching is slightly more complicated. Let $F$ be an arbitrary set in $\mathcal{F} \setminus \mathcal{Q}(m, s, \ell)$. Denote
    \[
    x = |F \cap [ms-c-1]| \qquad \text{and} \qquad y = |F \setminus [ms-c-1]|.
    \]
    By the definition of $\mathcal{Q}(m, s, \ell)$, we have \begin{equation}\label{eqxy}2x + y \leq 2m - 1.\end{equation}

    Assume that $x \leq m-c-1$. Since $x \geq 0$, we have $m \geq c + 1$, and therefore
    \[
    d(\mathcal{F}) = (m-1)c > mc-c - 1 \geq m(c-1).
    \]
    Thus, we may apply Lemma~\ref{l.deleteon_with_shifting} with $d = mc-c-1$ and $k=c-1$, and conclude that for any
    $
    X \in {[ms-c-1] \choose m-c-1}
    $
    the family $\mathcal{F}\cap {[ms-c-1]\setminus X \choose m}$ contains an $(s-1)$-matching. Since
    \[
    x = |F \cap [ms-c-1]| \leq m - c - 1,
    \]
    it follows that $\mathcal{F}\cap {[ms-c-1]\setminus F \choose m}$ contains an $(s-1)$-matching. Adding $F$ to this matching, we obtain an $s$-matching in $\mathcal{F}$, a contradiction with $\nu(\m F)<s$.

    Thus, we may assume that $x \geq m - c$. Construct a random $s$-matching $\pi$ in
    \[
    \Bigl(\mathcal{F} \cap {[ms-c-1] \choose m}\Bigr) \cup \Bigl({[ms-c-1] \choose m-1} \times {[ms-c, n] \choose 2}\Bigr) \cup \{F\}
    \]
    by the following procedure.
    \begin{itemize}
        \item Add $F$ to $\pi$.
        \item Choose
        $
        X \in {[ms-c-1] \setminus F \choose (m-1)(x+c+1-m)}
        $
        uniformly at random. Note that $x+c+1-m > 0$, since $x \geq m - c$.
        \item By Lemma~\ref{l.deleteon_with_shifting}, with $d = mc-c-1$ and $k = m-x-2$, for any
        $
        Y \in {[ms-c-1] \choose d-km}
        $
        the family $\mathcal{F} \cap {[ms-c-1] \setminus Y \choose m}$ contains an $(\ell+m-x-2)$-matching. We have $d-km=(m-1)(x+c+1-m) + x$. We apply the lemma to
        $
        Y = X \cup (F \cap [ms-c-1])
        $
        and obtain an $(\ell+m-x-2)$-matching in
        \[
        \mathcal{F} \cap {[ms-c-1] \setminus (X \cup F) \choose m}.
        \]
        Add this matching to $\pi$.
        \item Choose a random $(x+c+1-m)$-matching in
        \[
        {X \choose m-1} \times {[ms-c, n] \setminus F \choose 2}
        \]
        and add it to $\pi$. Note that such matchings exist, since $|X|=(m-1)(x+c+1-m)$ and
        \[
        |[ms-c, n] \setminus F| = 2c+1-y \overset{\eqref{eqxy}}{\geq} 2(x+c+1-m).
        \]
    \end{itemize}

    Note that $\pi$ contains
    \[
    1+(\ell+m-x-2)+(x+c+1-m) = s
    \]
    sets, and all sets except those added in the last step belong to $\mathcal{F}$. Let $\xi$ be the number of sets from $\overline{\mathcal{F}}$ in $\pi$. On the one hand, $\xi \geq 1$, since $\pi$ is an $s$-matching and $\mathcal{F}$ does not contain an $s$-matching. Therefore, $\mathbb{E}\xi \geq 1$. On the other hand, for any set
    \[
    H \in \m G:={[ms-c-1] \setminus F \choose m-1} \times {[ms-c, n] \setminus F \choose 2}
    \]
    we have
    \[
    \mathbb{P}[H \in \pi] = (x+c+1-m){ms-c-1-x \choose m-1}^{-1}{2c+1-y \choose 2}^{-1}.
    \]
    Therefore,
    \[
    1\le \mathbb{E}\xi = |\overline{\mathcal{F}} \cap \m G|
    \cdot (x+c+1-m){ms-c-1-x \choose m-1}^{-1}{2c+1-y \choose 2}^{-1}.
    \]
    Combining these inequalities, we get
    \[
    \begin{split}
        y_{\mathcal{F}}(m+1)
        &\geq |\overline{\mathcal{F}} \cap \m G| \\
        &\geq (x+c+1-m)^{-1}{ms-c-1-x \choose m-1}{2c+1-y \choose 2} \\
        &\geq (x+c+1-m)^{-1}{ms-c-1-x \choose m-1}{2(x+c+1-m) \choose 2} \\
        &= (2x+2c+1-2m){n-2c-1-x \choose m-1} \\
        &\geq {n-2c-1-x \choose m-1} \\
        &= {n \choose m - 1} + O(s^{m - 2}).
    \end{split}
    \]
\end{proof}

\begin{proof}[Proof of Theorem~\ref{t.layer_stability}]
If $d(\mathcal{F}) > mc$, then, by Lemma~\ref{l.deleteon_with_shifting} with $d=mc$ and $k=c$, the family $\mathcal{F} \cap {[ms] \choose m}$ contains an $s$-matching, and therefore $\mathcal{F}$ contains an $s$-matching, a contradiction. Thus, $d(\m F)\le mc.$ We consider three cases:
\[
d(\mathcal{F}) < (m-1)c,\qquad d(\mathcal{F}) = (m-1)c,\qquad \text{and} \qquad (m-1)c < d(\mathcal{F}) \leq mc.
\]

If $d(\mathcal{F}) < (m-1)c$, then by Lemma~\ref{l.d_to_ym} we get
\[
y_{\mathcal{F}}(m) \geq ((m+1)c-((m-1)c-1)+1){n \choose m - 1} + O(s^{m - 2}) = (2c+2){n \choose m - 1} + O(s^{m - 2}).
\]

If $d(\mathcal{F}) = (m-1)c$ and $\mathcal{F} \nsubseteq \mathcal{Q}(m,s,\ell)$, then by Lemma~\ref{l.d_to_ym} we get
\[
y_{\mathcal{F}}(m) \geq (2c+1){n \choose m - 1} + O(s^{m - 2}),
\]
and by Lemma~\ref{l.stability_d_eq_mc-c} we get
\[
y_{\mathcal{F}}(m+1) \geq {n \choose m - 1} + O(s^{m-2}).
\]
Therefore,
\[
y_{\mathcal{F}}(m) + y_{\mathcal{F}}(m + 1) \geq (2c+2){n \choose m - 1} + O(s^{m - 2}).
\]

The case $(m-1)c < d(\mathcal{F}) \leq mc$ requires a slightly longer analysis. Since $d(\mathcal{W}(m, s, \ell)) = mc$, it is natural that well have to consider  the case $c = 1$ and $d(\mathcal{F}) = mc$ separately. 
%First, we check that Lemmas~\ref{l.d_to_ym} and~\ref{l.d_geq_mc-c+1} imply
%\[
%y_{\mathcal{F}}(m) + y_{\mathcal{F}}(m + 1) \geq (2c+2){n \choose m - 1} + O(s^{m - 2}),
%\]
If $d(\mathcal{F}) = mc-\alpha$, $\alpha\ge 0$, then by Lemma~\ref{l.d_to_ym} we get
{\small \begin{align} \label{eq:tmp1}
    y_{\mathcal{F}}(m) \geq ((m+1)c-d(\m F)+1){n \choose m - 1} + O(s^{m - 2}) = (c+\alpha+1){n \choose m - 1} + O(s^{m - 2}),
\end{align}}
and by Lemma~\ref{l.d_geq_mc-c+1} we get
\[
y_{\mathcal{F}}(m+1) \geq (2c-1){n \choose m - 1} + O(s^{m-2}).
\]
Therefore,
\[
    y_{\mathcal{F}}(m) + y_{\mathcal{F}}(m + 1) \geq (3c+\alpha){n \choose m - 1}\ge (2c+2){n \choose m - 1} + O(s^{m - 2}),
\]
where the second inequality is valid unless $c = 1$ and $d(\mathcal{F}) = mc$.

Finally, we prove that if $d(\mathcal{F}) = mc$ and $c=1$, then $\mathcal{F} \subset \mathcal{W}(m, s, s-1)$. Assume the contrary, that is, that there exists $F \in \mathcal{F}$ such that $|F \cap [ms - 1]| \leq m - 1$. Let $X$ be an arbitrary $(m-1)$-element subset of $[ms - 1]$ containing $F \cap [ms - 1]$. By Lemma~\ref{l.deleteon_with_shifting} with $d = mc - 1$ and $k = c - 1$, the family $\mathcal{F}\cap {[ms - 1] \setminus X \choose m}$ contains an $(s-1)$-matching. Adding $F$ to this matching, we obtain an $s$-matching in $\mathcal{F}$, contradicting $\nu(\mathcal{F}) < s$.
\end{proof}

\section{Proof of Theorem~\ref{t.main}} \label{s.no_m-1}

In this section we prove the following stability version of Theorem~\ref{t.main} for shifted families.

\begin{theorem} \label{t.main_stability}
    Let $n = ms + c$, $m \geq 2$. Let $\mathcal{F} \subset 2^{[n]}$ be a shifted family and $\nu(\mathcal{F}) < s$. We have   \[
    |\overline{\mathcal{F}}| \geq (2c+3){n \choose m - 1} + O(s^{m - 2}),
    \] 
    provided one of the following holds.
    \begin{itemize}
        \item $c \geq 2$ and $\mathcal{F} \nsubseteq \mathcal{Q}(m, s, \ell)$;
        \item $c = 1$, $\mathcal{F} \nsubseteq \mathcal{Q}(m, s, \ell)$, and $\mathcal{F} \nsubseteq \mathcal{W}(m, s, \ell)$.
    \end{itemize}
\end{theorem}

It is easy to see that this implies  Theorem~\ref{t.main}. Indeed, we have
\[
|\overline{\mathcal{Q}(m, s, \ell)}| = (2c+2){n \choose m - 1} + O(s^{m-2}),
\]
and $|\mathcal{Q}(m, s, \ell)| = |\mathcal{W}(m, s, \ell)|$ for $c=1$. Thus, Theorem~\ref{t.main_stability} together with the fact that the value of $e(n, s)$ is attained by a shifted family proves Theorem~\ref{t.main}.

%First, we note that Theorem \ref{t.main_stability} implies Theorem \ref{t.main}.

%\textit{Proof of Theorem \ref{t.main}, using Theorem \ref{t.main_stability}}.
%It is well-known \cite{Frankl_Shifting} that $e(n, s)$ is attained in some shifted family $\mathcal{F}$ and either $\mathcal{F} \subset \mathcal{Q}(m, s, \ell)$ (and, therefore, $e(n, s) = |\mathcal{F}| \leq |\mathcal{Q}(m, s, \ell)|$), or $$|\overline{\mathcal{F}}| \geq (2c+3){n \choose m - 1} + O(s^{m - 2}) > (2c+2){n \choose m - 1} + O(s^{m - 2}) = |\overline{\mathcal{Q}(m, s, \ell)}|$$ for $s \geq s_0(m,c)$. \qed

We prove Theorem~\ref{t.main_stability} in two steps. First, we derive from Theorem~\ref{t.layer_stability} an approximate version of Theorem~\ref{t.main}. Then we use this approximate version to prove Theorem~\ref{t.main_stability} for families that contain at least one $(m-1)$-element set.

\begin{lemma} \label{l.approximate_main}
    Let $n = ms + c$, $m \geq 2$. Let $\mathcal{F} \subset 2^{[n]}$ be a family with $\nu(\mathcal{F}) < s$. Then
    \[
    |\overline{\mathcal{F}}| \geq (2c+2){n \choose m - 1} + O(s^{m - 2}).
    \]
\end{lemma}

\begin{proof}
    We may w.l.o.g. assume that $\m F$ is shifted. If $\mathcal{F} \subset \mathcal{Q}(m, s, \ell)$ or $\mathcal{F} \subset \mathcal{W}(m, s, \ell)$, then
    \[
    |\overline{\mathcal{F}}| \geq |\overline{\mathcal{Q}(m, s, \ell)}| = (2c+2){n \choose m - 1} + O(s^{m - 2}).
    \]
    Otherwise, by Theorem~\ref{t.layer_stability},
    \[
    |\overline{\mathcal{F}}| \geq y_{\mathcal{F}}(m) + y_{\mathcal{F}}(m + 1) \geq (2c+2){n \choose m - 1} + O(s^{m - 2}).
    \]
\end{proof}

\begin{proof}[Proof of Theorem~\ref{t.main_stability}]
    If $\mathcal{F}$ does not contain any $(m-1)$-element set, then $y_{\mathcal{F}}(m-1) = {n \choose m - 1}$. By Theorem~\ref{t.layer_stability} we have
    \[
    y_{\mathcal{F}}(m) + y_{\mathcal{F}}(m + 1) \geq (2c+2){n \choose m - 1} + O(s^{m - 2}).
    \]
    Therefore,
    \[
    |\overline{\mathcal{F}}| \geq (2c+3){n \choose m - 1} + O(s^{m - 2}).
    \]

    If $\mathcal{F}$ contains an $(m-1)$-element set, let $F$ be an arbitrary $(m-1)$-element set in $\mathcal{F}$ and let
    \[
    \mathcal{F}' = \mathcal{F} \cap 2^{[n]\setminus F}.
    \]
    Since $\mathcal{F}'$ is a  family of subsets of an $((s-1)m+(c+1))$-element set that does not contain an $(s-1)$-matching, we may apply Lemma~\ref{l.approximate_main} to $\mathcal{F}'$ and get
    \[
    |\overline{\mathcal{F}}| \geq |\overline{\mathcal{F}'}| \geq (2c+4){n - m + 1 \choose m - 1} + O((s-1)^{m - 2}) = (2c+4){n \choose m - 1} + O(n^{m - 2}).
    \]
\end{proof}

\end{document}